\documentclass[11pt,a4paper,draft]{amsart}
\usepackage[latin1]{inputenc}
\usepackage[T1]{fontenc}
\usepackage{textcomp}
\usepackage{lmodern}
\usepackage{amsrefs}

\usepackage{tikz}

\usepackage{xspace}
\usepackage{enumitem}
\setenumerate[1]{label=(\roman{*}), ref=(\roman{*})}
\setenumerate[2]{label=(\alph{*}), ref=(\alph{*})}

\theoremstyle{plain}
\newtheorem{thm}{Theorem}[section]
\newtheorem*{thm*}{Theorem}
\newtheorem*{lem*}{Lemma}
\newtheorem{cor}[thm]{Corollary}
\newtheorem{prop}[thm]{Proposition}
\newtheorem{lem}[thm]{Lemma}

\newtheorem{quest}[thm]{Question}

\theoremstyle{definition}
\newtheorem{defn}[thm]{Definition}
\newtheorem{exmp}[thm]{Example}

\theoremstyle{remark}
\newtheorem{rem}[thm]{Remark}

\DeclareMathOperator{\linspan}{span}

\DeclareMathOperator{\supp}{supp}

\newcommand{\tn}{|\mspace{-1mu}|\mspace{-1mu}|}

\begin{document}
\title{Almost square Banach spaces}

\author[T.~A.~Abrahamsen]{Trond~A.~Abrahamsen}
\address{Department of Mathematics, Agder University, Servicebox 422,
4604 Kristiansand, Norway.}
\email{Trond.A.Abrahamsen@uia.no}

\author[J.~Langemets]{Johann Langemets}
\address{Institute of Mathematics, University of Tartu, J. Liivi 2,
50409 Tartu, Estonia}
\email{johann.langemets@ut.ee}
\thanks{The research of J. Langemets was supported by
Estonian Science Foundation Grant 8976,  Estonian Targeted
Financing Project SF0180039s08, and Estonian Institutional
Research Project IUT20-57.}

\author[V.~Lima]{Vegard Lima}
\address{{\AA}lesund University College, Postboks 1517, 6025
  {\AA}lesund, Norway.}
\email{Vegard.Lima@gmail.com}

\date{\today}

\begin{abstract}
  We single out and study a natural class of Banach spaces
  -- almost square Banach spaces.
  In an almost square space we can find, given
  a finite set $x_1,x_2,\ldots,x_N$ in the unit sphere,
  a unit vector $y$ such that $\|x_i-y\|$ is almost one.
  These spaces have duals that
  are octahedral and finite convex combinations of slices
  of the unit ball of an almost square space have diameter 2.
  We provide several examples and
  characterizations of almost square spaces.
  We prove that non-reflexive spaces which are
  M-ideals in their biduals are almost square.

  We show that every separable space containing a copy of $c_0$
  can be renormed to be almost square.
  A local and a weak version of almost square spaces
  are also studied.
\end{abstract}

\maketitle

\section{Introduction}
Let $X$ be a Banach space with unit ball $B_X$,
unit sphere $S_X$, and dual space $X^*$.

\begin{defn}
  We will say that a Banach space $X$ is
  \begin{enumerate}
  \item
    \emph{locally almost square} (LASQ) if for every $x \in S_X$
    there exists a sequence $(y_n) \subset B_X$ such that $\|x \pm y_n\| \to 1$
    and $\|y_n\| \to 1$.
  \item
    \emph{weakly almost square} (WASQ) if for
    every $x \in S_X$ there exists a sequence $(y_n) \subset B_X$ such that $\|x
    \pm y_n\| \to 1$, $\|y_n\| \to 1$ and $y_n \to 0$ weakly.
  \item
    \emph{almost square} (ASQ) if for every finite subset
    $(x_i)_{i=1}^N \subset S_X$ there exists a sequence $(y_n) \subset B_X$ such
    that $\|x_i \pm y_n\| \to 1$ for every $i=1,2,\ldots,N$
    and $\|y_n\| \to 1$.
  \end{enumerate}
\end{defn}
Obviously WASQ implies LASQ, but
it is not completely obvious that ASQ implies WASQ.
This will be shown in Theorem~\ref{thm-asiswas}.
In the language of
Sch\"affer \cite{MR0467256}*{p.~31}
a Banach space $X$ is LASQ if and only if no $x \in S_X$
is \emph{uniformly non-square}
(see also \cite[Proposition~2.2]{Kub13}).

The above definition was inspired by
the following characterizations of octahedral norms
shown by Haller, Langemets, and P{\~o}ldvere in \cite{HLP}
(see Proposition~2.1, Proposition~2.4, and Lemma~3.1).
\begin{prop}\label{defn:oct}
  A Banach space $X$ is said to be
  \begin{enumerate}
  \item \emph{locally octahedral} if for every $x \in S_X$ and every
    $\varepsilon > 0$ there exists $y \in S_X$ such that $\| x \pm y\|
    \ge 2-\varepsilon$.
  \item \emph{weakly octahedral} if for every finite subset
    $(x_i)_{i=1}^N \subset S_X$, every $x^*\in B_{X^*}$, and every
    $\varepsilon>0$ there exists $y \in S_X$ such that $\|x_i + ty\| \ge
    (1-\varepsilon)(|x^*(x_i)| + t)$ for all $i=1,2,\ldots,N$ and
    $t>0$.
  \item \emph{octahedral} if for every finite subset
    $(x_i)_{i=1}^N \subset S_X$ and every $\varepsilon>0$ there exists $y \in
    S_X$ such that $\|x_i \pm y\| \ge 2-\varepsilon$ for all
    $i=1,2,\ldots,N$.
  \end{enumerate}
\end{prop}

Our main interest in the properties LASQ, WASQ, and ASQ come from
their relation to diameter two properties.
Let $X$ be a Banach space.
Recall that a \emph{slice} of $B_X$ is a set of the form
\begin{equation*}
  S(x^*,\alpha) = \{
  x \in B_X : x^*(x) > 1 - \alpha
  \},
\end{equation*}
where $x^* \in S_{X^*}$ and $\alpha > 0$.
According to \cite{ALN01} $X$ has the
\emph{local diameter $2$ property (LD2P)} if every slice of
$B_X$ has diameter $2$,
$X$ has the \emph{diameter $2$ property (D2P)} if every nonempty
relatively weakly open subset of $B_X$ has diameter $2$,
and $X$ has the \emph{strong diameter $2$ property (SD2P)}
if every finite convex combination of slices of $B_X$ has diameter $2$.
The following theorem is shown in
\cite{HLP} (see Theorems~2.3, 2.7, and 3.3).
\begin{thm}
  Let $X$ be a Banach space. Then
  \begin{enumerate}
  \item $X$ has the LD2P
    if and only if $X^*$ is locally octahedral.
  \item $X$ has the D2P
    if and only if $X^*$ is weakly octahedral.
  \item $X$ has the SD2P
    if and only if $X^*$ is octahedral.
  \end{enumerate}
\end{thm}
The connection between the SD2P and octahedrality
has also been studied in \cite{BGLPRZ-Oct}.

The starting point of this paper was the observation
by Kubiak that if $X$ is LASQ then $X$ has the LD2P
and similarly if $X$ is WASQ then $X$ has the D2P
(see Propositions~2.5 and ~2.6 in \cite{Kub13}).
The basic idea from Kubiak's proof can be used
with a result from \cite{ALN01} to show that ASQ spaces have the SD2P,
but we will give a shorter self-contained
proof that $X^*$ is octahedral whenever $X$ is ASQ
in Proposition~\ref{prop:asq-dual-oct}.

It is known that the three diameter $2$ properties
are different (see \cite{BGLPRZ},
\cite{ABGLP13}, and \cite{HL}).
A natural question is whether
LASQ, WASQ, and ASQ are different properties.
A consequence of Example~\ref{exmp:wasqNEQasq}
is that $L_1[0,1]$ is a WASQ space which is not ASQ.

We now give a short outline of the paper.
Section~\ref{sec:characterizations} starts with a few characterizations
of LASQ and ASQ. In Lemma~\ref{lem:epsiso2} we show that
ASQ spaces have to contain almost isometric copies of $c_0$.
This in turn is used to prove Theorem~\ref{thm-asiswas},
which shows that ASQ implies WASQ.
The final result in this section
is Theorem~\ref{thm:c0-asq}, where
we show that every separable Banach space that contains a
copy of $c_0$ can be equivalently renormed to be ASQ.

In Section~\ref{sec:examples1} we will give examples of spaces which
are LASQ, WASQ, and ASQ.

In Section~\ref{sec:examples2} we show that
non-reflexive spaces which are M-ideals in their biduals are ASQ
(see Theorem~\ref{thm:mid-sc-sq}). However, the
class of ASQ spaces is much bigger than
the class of spaces that are M-ideals in their
biduals (see the discussion following
Corollary~\ref{cor:mem-asq} and Example~\ref{exmp:asqNotMid}).

In Section~\ref{sec:stability} we study the stability
of both (local/weak) octahedrality and (L/W)ASQ
when forming absolute sums of Banach spaces.
We show that local and weak octahedral,
LASQ, and WASQ spaces have nice stability properties
(see Propositions~\ref{prop:stab-was}, \ref{prop: stab loct and lasq},
and \ref{prop:stab-oct}) but the
situation is different for ASQ.
For $1 \le p < \infty$ the $\ell_p$-sum of
two Banach spaces is never ASQ.
By Proposition~3.10 in \cite{HLP}
an $\ell_p$-sum of two Banach spaces can only
be octahedral if $p=1$ or $p=\infty$.

In Section~\ref{sec:connection-with-ip}
we connect ASQ with the intersection property
of Behrends and Harmand. We show that
ASQ spaces fail the intersection property
and give a quantitative version of this fact
in Proposition~\ref{thm:epsfailsIPandas}.
We also give an example of a space that fails
the intersection property and is not LASQ.

We follow standard Banach space notation as used
in e.g. \cite{AlKa}. We consider real Banach spaces only.

\section{Characterizations}
\label{sec:characterizations}
The following characterization of LASQ and ASQ
is clear from the definition.
Note that since we have finitely many vectors to play
with in the definition of ASQ we may drop the plus-minus sign.

\begin{prop}\label{prop:lasNas-1dim}
  Let $X$ be a Banach space.

  $X$ is LASQ if and only if for every $x \in S_X$
  and every $\varepsilon > 0$ there exists $y \in S_X$
  such that $\|x \pm y\| \le 1 + \varepsilon$.

  $X$ is ASQ if and only if for every finite subset
  $(x_i)_{i=1}^N \subset S_X$
  and every $\varepsilon > 0$ there exists $y \in S_X$
  such that $\|x_i - y\| \le 1 + \varepsilon$.
\end{prop}

The following lemma is surely well-known but we
include it for easy reference.
\begin{lem}\label{lem:c0-lookalike}
  Assume $x,y \in S_X$ such that
  $\|x \pm y\| \le 1 + \varepsilon$, then
  \begin{equation*}
    (1-\varepsilon)\max(|\alpha|,|\beta|)
    \le \|\alpha x + \beta y\|
    \le
    (1+\varepsilon)\max(|\alpha|,|\beta|)
  \end{equation*}
  for all scalars $\alpha$ and $\beta$.
\end{lem}

\begin{proof}
  Let $M = \max(|\alpha|,|\beta|)$.
  We need to show that
  \begin{equation*}
    1-\varepsilon
    \le \|\frac{\alpha}{M} x + \frac{\beta}{M} y\|
    \le
    1+\varepsilon.
  \end{equation*}
  It is enough to show
  \begin{equation*}
    1-\varepsilon
    \le \|\lambda x \pm y\|
    \le
    1+\varepsilon
  \end{equation*}
  for all $0 < \lambda \le 1$.
  By the triangle inequality, we have that
  $1-\varepsilon \le \|x \pm y\| \le 1+ \varepsilon$
  whenever $x,y \in S_X$ satisfies $\|x \pm y\| \le 1+\varepsilon$.
  It follows that
  \begin{equation*}
    \|\lambda^{-1} y + x\|
    = \|(1+\lambda^{-1})y - (y-x)\|
    \ge (1+\lambda^{-1}) - \|x-y\|
    \ge \lambda^{-1} - \varepsilon
  \end{equation*}
  since $\|x-y\| \le 1 + \varepsilon$.
  Hence $\|\lambda x + y\|
  \ge 1-\varepsilon \lambda \ge 1 - \varepsilon$.

  Also
  \begin{equation*}
    \|\lambda^{-1} y + x\|
    = \|(\lambda^{-1}-1)y + (y+x)\|
    \le (\lambda^{-1}-1) + 1+\varepsilon
    = \lambda^{-1} + \varepsilon
  \end{equation*}
  and hence $\|\lambda x + y \| \le 1 + \varepsilon \lambda
  \le 1 + \varepsilon$.
\end{proof}

\begin{cor}\label{cor:las-2dimellinfty}
  If $X$ is LASQ, then $X$ contains
  almost isometric copies of $\ell_\infty^2$.
\end{cor}

For ASQ Banach spaces we can say even more.

\begin{thm}\label{thm:as-findim}
  Let $X$ be a Banach space.
  If $X$ is ASQ then for every
  finite dimensional subspace $E \subset X$
  and every $\varepsilon > 0$
  there exists $y \in S_X$ such that
  \begin{equation*}
    (1-\varepsilon)\max(\|x\|,|\lambda|)
    \le \|x + \lambda y\|
    \le
    (1+\varepsilon)\max(\|x\|,|\lambda|)
  \end{equation*}
  for all scalars $\lambda$ and all $x \in E$.

  Moreover, given a finite dimensional subspace $F \subset X^*$
  we may choose the above $y$ so that $|f(y)| < \varepsilon \|f\|$
  for every $f \in F$.
\end{thm}

It is clear from Proposition~\ref{prop:lasNas-1dim}
that the above theorem is actually a characterization of ASQ.

\begin{proof}
  Let $E$ be a finite dimensional subspace of $X$ and
  let $\varepsilon > 0$.
  Find an $\varepsilon/2$-net $(x_i)_{i=1}^N$ for $S_E$.
  Choose $y \in S_X$ such that
  $\|x_i \pm y\| \le 1 + \varepsilon/2$.
  By the triangle inequality $\|2a\| - \|a-b\|
  \le \|a+b\|$ and hence $\|x_i \pm y\| \ge 1 - \varepsilon/2$.

  Let $x \in S_E$. Find $i$ such that $\|x_i - x\| \le \varepsilon/2$.
  Then
  \begin{equation*}
    \|x \pm y\| \le \|x_i \pm y\| + \|x-x_i\| \le 1 + \varepsilon
  \end{equation*}
  and thus also $\|x \pm y\| \ge 1 - \varepsilon$.
  Hence by using Lemma~\ref{lem:c0-lookalike} we get
  \begin{equation*}
    (1-\varepsilon)\max(\|x\|,|\lambda|)
    \le \|x + \lambda y\|
    \le
    (1+\varepsilon)\max(\|x\|,|\lambda|)
  \end{equation*}
  for all scalars $\lambda$ and all $x \in E$.

  For the moreover part let $F \subset X^*$ be
  a finite dimensional subspace and let
  $(f_i)_{i=1}^M \subset S_F$ be an $\varepsilon/2$-net.
  For each $i$ choose $z_i \in S_X$ with $f_i(z_i) > 1-\varepsilon/4$.
  Let $E' = \linspan\{E,(z_i)_{i=1}^M\}$ and use the
  first part of the proof to find $y \in S_X$ such that
  \begin{equation*}
    (1-\varepsilon/4)\max(\|x\|,|\lambda|)
    \le \|x + \lambda y\|
    \le
    (1+\varepsilon/4)\max(\|x\|,|\lambda|)
  \end{equation*}
  for all scalars $\lambda$ and all $x \in E'$.

  Since $|f_i(z_i \pm y)| \le \|z_i \pm y\| \le 1+\varepsilon/4$
  we get
  \begin{align*}
    -\varepsilon/2 &= 1-\varepsilon/4-(1+\varepsilon/4)
    \le f_i(z_i)-f_i(z_i-y) = f_i(y) \\
    &\le f_i(z_i + y) - f_i(z_i)
    \le 1+ \varepsilon/4 - 1 + \varepsilon/4 = \varepsilon/2.
  \end{align*}
  so that $|f_i(y)| < \varepsilon/2$.
  Thus for every $f \in S_F$, and for some $i$, we have
  $|f(y)| \le |(f-f_i)(y)| + |f_i(y)| \le \varepsilon$.
\end{proof}

\begin{prop}\label{prop:asq-dual-oct}
  If $X$ is ASQ, then $X^*$ is octahedral.
\end{prop}

\begin{proof}
  Let $x_1^*,x_2^*,\ldots,x_n^* \in S_{X^*}$ and $\varepsilon > 0$.

  Find $x_1,x_2,\ldots,x_n \in S_X$ such that $x_i^*(x_i) > 1-\varepsilon$.
  Using Theorem~\ref{thm:as-findim} find
  $y \in S_X$ such that $\|x_i \pm y\| \le 1+ \varepsilon$
  and $|x_i^*(y)| < \varepsilon$.

  Find $y^* \in S_{X^*}$ such that $y^*(y) = 1$.
  Then
  \begin{equation*}
    1 + \varepsilon \ge \|x_i \pm y\|
    \ge \pm y^*(x_i) + y^*(y)
    = \pm y^*(x_i) + 1
  \end{equation*}
  and thus $|y^*(x_i)| \le \varepsilon$.
  Now
  \begin{align*}
    \|x_i+y\|\|x_i^* + y^*\|
    &\ge
    x_i^*(x_i) + x_i^*(y) + y^*(x_i) + y^*(y) \\
    &> 1 - \varepsilon - 2\varepsilon + 1
  \end{align*}
  and hence
  \begin{equation*}
    \|x_i^* + y^*\| > \frac{2-3\varepsilon}{1+\varepsilon}
  \end{equation*}
  which shows that $X^*$ is octahedral by
  Proposition~2.1 in \cite{HLP}.
\end{proof}

Repeated use of Theorem~\ref{thm:as-findim} gives the following lemma.

\begin{lem}\label{lem:epsiso2}
  If $X$ is ASQ, then for every finite dimensional subspace $E$ of
  $X$ and every $\varepsilon > 0$ there exists a subspace $Y$ of $X$
  which is $\varepsilon$-isometric to $c_0$ such that
  $E \oplus Y$ is $\varepsilon$-isometric to $E \oplus_\infty c_0$.
\end{lem}

\begin{proof}
  Find sequence $(\varepsilon_n) \subset \mathbb{R}^+$ such
  that $\prod_{n=1}^\infty (1+\varepsilon_n) < 1 + \varepsilon$
  and  $\prod_{n=1}^\infty (1-\varepsilon_n) > 1 - \varepsilon$.
  Using Theorem~\ref{thm:as-findim} we
  inductively choose a sequence $(y_n) \subset S_X$ such that
  \begin{equation*}
    (1-\varepsilon_n)\max\{\|e\|,|\lambda|\}
    \le
    \|e + \lambda y_n\|
    \le
    (1+\varepsilon_n)\max\{\|e\|,|\lambda|\}
  \end{equation*}
  for every $e \in \linspan\{E,(y_i)_{i=1}^{n-1}\}$ and every $\lambda
  \in \mathbb{R}$.
  Then $Y = \overline{\linspan\{(y_n)\}}$ is
  $\varepsilon$-isometric
  to $c_0$ and defining $S: E \oplus_\infty c_0 \to E \oplus Y$
  by $S(e,a) = e + Ta$ where $T:c_0 \to Y$ is the
  $\varepsilon$-isometry.
  We have
  \begin{align*}
    \|S(e, \sum_{n=1}^N a_ne_n)\| & = \|e + \sum_{n=1}^N a_ny_n\|
    \le (1 + \varepsilon_N)\max\{\|e + \sum_{n=1}^{N-1} a_ny_n\|,
    |a_N|\} \\
    &\le \cdots
    \le \prod_{n=1}^N(1 +
    \varepsilon_n)\max\{\|e\|,|a_1|,|a_2|,\ldots,|a_N|\}\\
    & < (1 + \varepsilon)\|(e, \sum_{n=1}^N a_ne_n)\|,
  \end{align*}
  and similarly $\|S(e, \sum_{n=1}^N a_ne_n)\| > (1 -
  \varepsilon)\|(e, \sum_{n=1}^N a_ne_n)\|$. Thus $S$ must be an
  $\varepsilon$-isometry onto $E \oplus Y$ since $T$ is onto $Y$.
\end{proof}

\begin{rem}\label{rem:asymp_iso_copies_c0}
  In Proposition~6 in \cite{MR1814162} Pfitzner showed
  that $M$-embedded spaces contain an asymptotically
  isometric copy of $c_0$ using the local characterization
  of $M$-ideals. If $X$ is ASQ and we use Theorem~\ref{thm:as-findim}
  in Pfitzner's proof we get that $X$
  actually contains an asymptotically isometric copy of $c_0$,
  and hence $X^*$ contains an asymptotically isometric copy of
  $\ell_1$ (see Theorem~2 in \cite{DowJLT}).
\end{rem}

A consequence of Lemma~\ref{lem:epsiso2} is that
the sequence $(y_n)$ in the definition of ASQ may be
chosen to be weakly null. This enables us to
connect the ASQ and WASQ properties.

\begin{thm}\label{thm-asiswas}
  If a Banach space $X$ is ASQ then
  for every $x_1,x_2,\ldots,x_N \in S_X$ there exists $(y_n) \subset B_X$
  such that $\|x_i \pm y_n\| \to 1$ for all $i$,
  $y_n \to 0$ weakly, and $\|y_n\| \to 1$.

  In particular, ASQ implies WASQ.
\end{thm}

\begin{proof}
  Let $x_1,x_2,\ldots,x_N \in S_X$ and $E = \linspan\{(x_i)_{i=1}^N\}$,
  and choose a sequence $(y_n) \subset S_X$ as in the proof of
  Lemma~\ref{lem:epsiso2}.
  Let $Z = E \oplus_\infty c_0$ and $z_i = (x_i, 0) \in Z$.
  Since the standard basis $(e_n)_{n=1}^\infty \subset S_{c_0}$ is
  weakly null so is $w_n = (0, e_n)$ in $Z$.
  By Lemma~\ref{lem:epsiso2} there exists an $\varepsilon$-isometry $S$
  from $Z$ onto $E \oplus Y$ where
  $Y = \overline{\linspan\{(y_n)\}}$.
  The weak-weak continuity of $S$ shows that $y_n \to 0$ weakly
  in $E \oplus Y$ and hence also in $X$.

  By definition $S(e, \pm e_n) = e \pm y_n$ for every $e \in E$.
  Since
  \begin{equation*}
    (1 -\varepsilon_n)\max\{\|e\|, 1\} \le \|e \pm y_n\| \le (1 +
    \varepsilon_n)\max\{\|e\|, 1\}
  \end{equation*}
  for every $e \in E$, we in particular have
  $(1 - \varepsilon_n)\le \|x_i \pm y_n\| \le (1 +
  \varepsilon_n)$, so $\|x_i \pm y_n\| \to 1$.
\end{proof}

We know that every ASQ space contains $c_0$. Next we will show
that a separable Banach space containing $c_0$ can be equivalently renormed to be
ASQ. For separable spaces this improves \cite[Proposition~4.7]{ALN01} which says
that any Banach space containing $c_0$ can be equivalently renormed to
have the SD2P (see also
\cite[Proposition~2.6]{MR2170568} for the D2P case).

The proof of the following result is based on a renorming technique which
appears in \cite[Lemma~2.3]{BGLPRZ}.

\begin{thm}\label{thm:c0-asq}
  A separable Banach space can be equivalently renormed to
  be ASQ if and only if it contains a copy of $c_0$.
\end{thm}

\begin{proof}
  As an ASQ-space contains $c_0$, the ``only if'' part is clear.

  For the ``if'' part, first renorm $X$ to contain $c_0$ isometrically
  \cite[Lemma 8.1]{DGZ}. Denote by $\|\cdot\|$ the new norm on $X$.
  By Sobczyk's theorem there exists a projection $P$
  onto $c_0$ with $\|P\| \le 2$.
  Define
  \begin{equation*}
    \tn x \tn = \max\{\|P(x)\|, \|x - P(x)\|\}.
  \end{equation*}
  Then $\tn \cdot\tn $ is a norm on $X$
  which satisfies $\frac{1}{2}\|x\| \le \tn x\tn  \le 3\|x\|$.
  Also $\tn \cdot\tn$ extends the max norm $\|\cdot\|$ on $c_0$
  and we get that $c_0$ is an M-summand in $X$ and hence $X$ is ASQ.
  (See the proof of Proposition~\ref{prop:*asq-inftysum}.)
\end{proof}

It is clear from the proof that all Banach spaces containing
a complemented copy of $c_0$ can be renormed to be ASQ.
We do not know whether or not the same is true for a general Banach space.

\section{Examples}
\label{sec:examples1}
In this section we will provide examples of Banach spaces which are
LASQ, WASQ, and ASQ and spaces which are not.

By considering the  constant 1 function in
$\ell_\infty$, $C(K)$, or $L_\infty[0,1]$ it is
clear that neither of these spaces are LASQ.
It is also obvious that $c_0$ is ASQ.
In fact, we have the following.

\begin{exmp}\label{ex1:c_0}
  Given a sequence of Banach spaces $(X_i)$
  the $c_0$-sum $c_0(X_i)$ is both WASQ and ASQ.

  Let $(x_i)_{i=1}^N \subset S_{c_0(X_i)}$.
  By choosing $j_n$ large enough we may assume
  that $\|x_i(j_n)\|<\frac{1}{n}$ for all $i=1,2,\ldots,N$.
  Choosing $y_{j_n} \in S_{X_{j_n}}$ and defining
  $y_n \in c_0(X_i)$ by
  $y_n(j) = 0$ for all $j$ except $y_n(j_n) = y_{j_n}$
  we see that $\|x_i \pm y_n\| \to 1$ as $n \to \infty$.

  As an extreme example $c_0(L_1[0,1])$
  is ASQ and octahedral (and even has the Daugavet property).
\end{exmp}

\begin{exmp}\label{ex:property-m_infty}
  A separable Banach space $X$ has Kalton and Werner's
  \emph{property $(m_\infty)$} if
  \begin{equation*}
    \limsup_n \|x + y_n\| = \max(\|x\|,\limsup_n\|y_n\|)
  \end{equation*}
  for every $x \in X$ whenever $y_n \to 0$ weakly.
  If $X$ has property $(m_\infty)$ then
  $X$ is ASQ if and only if $X$ lacks the Schur property.

  However, if $X$ does not contain a copy of $\ell_1$,
  then by Theorem~3.5 in \cite{MR1324212}
  $X$ has property $(m_\infty)$ if and only if
  $X$ is almost isometric to a subspace of $c_0$.
  This is much stronger than ASQ,
  see Corollary~\ref{cor:mem-asq} below.
\end{exmp}

In \cite{MR1026841} Gao and Lau
considered the following parameter
\begin{equation*}
  G(X) = \sup \{ \inf \{
  \max\{ \|x+y\|,\|x-y\|
  \}, y \in S_X
  \}, x \in S_Y \}.
\end{equation*}
We see from Proposition~\ref{prop:lasNas-1dim}
that $X$ is LASQ if and only if $G(X) = 1$.
Gao and Lau showed that
$L_1[0,1]$ is LASQ while $L_p[0,1]$, $1<p\le \infty$,
and $\ell_p$, $1 \le p \le \infty$, are not.
In \cite{MR0228997} Whitley introduced
the \emph{thinness index}, $t(X)$, of a Banach space $X$.
It is not difficult to see that Whitley's original
definition of $t(X)$ is equivalent to
\begin{equation*}
  t(X) = \sup_{x_1,x_2,\ldots,x_N \in S_X} \inf_{y \in S_X}
  \max_i \|x_i - y\|.
\end{equation*}
We see that $t(X)=1$ if and only if $X$ is ASQ.

As noted above Gao and Lau have shown that
$L_1[0,1]$ is LASQ, in fact it is WASQ but it is not ASQ.
This is a special case of our next
example which concerns Ces{\`a}ro function spaces.
We will need a bit of Banach lattice notation
(for more see e.g. \cite{LiTz2}).

For an interval $I \subset \mathbb R$ by $L_0 (I)$ we denote the set
of all (equivalence classes of) real valued Lebesgue measurable
(finite almost everywhere) functions on $I$.
A \emph{Banach function lattice} is a Banach space
$E = E(I) \subset L_0 (I)$ such that
if $|f(x)| \le |g(x)|$ a.e. with $f \in L_0 (I)$ and $g \in E$,
then $f \in E$ and $\|f\| \le \|g\|$.
$E$ is \emph{order continuous} if for every $f \in E$
and every $0 \le f_n \le |f |$ a.e. such
that $f_n \downarrow 0$ a.e. we have that $\|f_n\| \downarrow 0$.
$E$ has the \emph{Fatou property}
if for any sequence $(f_n) \subset E $ and any
$f \in L_0 (I)$ such that $0 \le f_n \le f$ a.e., $f_n \uparrow f$
a.e., and $\sup_n \|f_n\| < \infty$ we have that $f \in E$ and $\|f\|
= \lim_n \|f_n\|$.

For $I = (0,l)$ with $0 < l \le \infty$ fixed and a
fixed weight $0 < \omega \in L_0(I)$ we can define a norm
on $L_0(I)$ for $1 \le p < \infty$ by
\begin{equation*}
  \|f\|_{C_{p,\omega}}
  =
  \left(
    \int_I \bigl( \omega(x) \int_{0}^x |f(t)| dt \bigr)^p dx
  \right)^{1/p}
\end{equation*}
The \emph{weighted Ces{\`a}ro function space} on $I$ is
defined by
\begin{equation*}
  C_{p,\omega} = C_{p,\omega}(I) = \{f \in L_0(I) : \|f\|_{p,\omega} < \infty\}.
\end{equation*}
It is known that $C_{p,\omega}$
in the natural pointwise order is a separable
order continuous Banach function lattice with the Fatou-property
(see \cite[Lemma 3.1]{MR2878472}).
Note that the space $C_{1, 1/x}[0,1]$ is isometrically
isomorphic to $L_1[0,1]$ (see e.g. \cite[p.~4293]{MR2431042}).

\begin{exmp}\label{exmp:wasqNEQasq}
  The Ces{\`a}ro function space $C_{p,\omega}$
  is WASQ but not ASQ.

  Kubiak \cite[Lemma~3.3]{Kub13} proved that $C_{p,\omega}$ is WASQ.
  In Proposition~\ref{prop:Cpw-no_c0} below we will
  show that $C_{p,\omega}$ does not contain $c_0$
  so by Lemma~\ref{lem:epsiso2} it is not ASQ.
\end{exmp}

Theorem~\ref{thm-asiswas} and
the example above shows that ASQ is strictly stronger than WASQ.
\begin{quest}
  Is WASQ strictly stronger than LASQ?
\end{quest}

\begin{prop}\label{prop:Cpw-no_c0}
  The space $C_{p,\omega}$ does not contain an isomorphic copy of $c_0$.
\end{prop}

\begin{proof}
  Let $(f_n)$ be an increasing norm bounded sequence in $C_{p,\omega}$.
  By \cite[Theorem 1.c.4]{LiTz2} it is enough to show that
  $(f_n)$ has a norm limit.
  If $(f_n)$ has a pointwise a.e. limit $f$, then
  it follows from the Fatou property that
  $f$ is in $C_{p,\omega}$. Let $g_n = f -
  f_n$. Then $0 \le g_n \le f-f_1$ and $g_n \downarrow
  0$. By order continuity we get that $\|f - f_n\| =
  \|g_n\| \to 0$ as wanted.

  It only remains to prove that the pointwise limit exists.
  $(f_n)$ increasing means that $f_n(x) \le f_{n+1}(x)$ for a.e. $x$.
  By completeness it is enough to show that
  $(f_n(x))$ is a bounded sequence for a.e. $x$.
  Assume not, i.e. that $\sup_n f_n(x) = \infty$ on a compact $A$
  of positive Lebesgue measure $\lambda(A) > 0$.
  Split $A$ into two parts $A_1$ and $A_2$
  with $\lambda(A_1) > 0$ and $\lambda(A_2) > 0$
  such that $x \le y$ for all $x \in A_1$ and $y \in A_2$.

  We know that
  \begin{equation*}
    K = \int_{A_2} w(x)^p \; dx > 0.
  \end{equation*}
  Let $S = \sup_n \|f_n\| < \infty$.
  Choose $k$ such that $S^p < M^p K$ where
  \begin{equation*}
    M = \int_{A_1} |f_k(t)| \; dt.
  \end{equation*}
  Then
  \begin{align*}
    S^p \ge \|f_k\|^p &=
    \int_I \left( w(x) \int_0^x |f_k(t)| dt \right)^p dx
    \ge
    \int_{A_2} \left( w(x) \int_0^x |f_k(t)| dt \right)^p dx \\
    &\ge
    \int_{A_2} \left( w(x) \int_{A_1} |f_k(t)| dt \right)^p dx
    =
    \int_{A_2} \left( w(x) M \right)^p dx
    = M^p K
  \end{align*}
  and we have a contradiction.
\end{proof}

Let us end this section by providing examples of ASQ, LASQ, and non-LASQ
from the class of Lindenstrauss spaces (i.e. the Banach spaces with duals
isometric to $L_1(\mu)$ for some positive measure $\mu$).
To see that the examples below are Lindenstrauss
cf. e.g. \cite[p.~80]{MR0179580}.
From Theorem~6.1~(14) in \cite{MR0179580} we see
that only Lindenstrauss spaces without extreme points can be LASQ.

Recall the thinness index $t(X)$ from the discussion
following Example~\ref{ex:property-m_infty}.
In \cite[Lemma~8]{MR0228997} Whitley showed
that for any compact Hausdorff we have $t(C(K))=2$,
while $t(C_0(K))=1$ if $K$ locally compact.
Hence $C(K)$-spaces are not ASQ, while $C_0(K)$-space are.
The idea is that every finite set of functions need
a common zero for the space to be ASQ.
Another example of the same idea is the following.

\begin{exmp}
  Let $K$ be a compact Hausdorff space and $\sigma: K \to K$
  involutory homeomorphism (i.e. $\sigma^2 = id_K$)
  with a non-isolated fixed point.
  Then $X = \{f \in C(K) : f(x) = -f(\sigma(x)) \;
  \mbox{for all}\; x \in K\}$ is ASQ.

  If $x_0$ is a fixed point for $\sigma$,
  then $f(x_0) = -f(\sigma(x_0)) = -f(x_0)$ for all $f \in X$.
  Hence given $f_1,f_2,\ldots,f_N \in S_X$ and $\varepsilon > 0$
  there is a neighborhood $U$ of $x_0$ in $K$ where
  $|f_i(x)| < \varepsilon$ for all $x \in U$ and
  $i=1,2,\ldots,N$.
  If we let $g \in S_X$ have support on $U \cup \sigma(U)$
  then $\|f_i \pm g\| \le 1 + \varepsilon$
  and thus $X$ is ASQ.
\end{exmp}

It is also not difficult to find LASQ subspaces
of $C(K)$-spaces.

\begin{exmp}
  Let $X = \{f \in C[0,1] : f(0)=-f(1)\}$.
  Then $X$ is LASQ but not ASQ.

  Since $f \in S_X$ has a zero in $[0,1]$ we can always find
  an interval where $|f(x)|$ is as small as we like.
  Any $g \in S_X$ with support in this interval
  has $\|f \pm g\|$ small so $X$ is LASQ.

  To see that $X$ is not ASQ let $f_1$ be any
  function in $S_X$ which is equal to $1$ on $[0,\frac{1}{2}]$
  and let $f_2(x) = f_1(1-x)$. Then $\max_i\|f_i \pm g\| = 2$
  for any $g \in S_X$.
\end{exmp}

\begin{rem}
  It is clear that in the above LASQ and ASQ examples
  we can construct a bounded sequence $(g_n)$ in the subspace $X$
  of $C(K)$ such that $g_n \to 0$ pointwise in $C(K)$.
  Thus $g_n \to 0$ weakly in $C(K)$,
  see e.g. Theorem~1 in \cite[p.~66]{Die-Seq},
  and hence must also be weakly null in $X$.
  It follows that these examples are also WASQ.
\end{rem}

\section{M-embedded spaces}
\label{sec:examples2}
In this section we will show that all M-embedded spaces are ASQ.
We start by recalling some definitions.

A subspace $Y$ of a Banach space $X$ is an \emph{ideal} in
$X$ if the annihilator $Y^{\perp}$ is the kernel of a norm one
projection on $X^*$.
The subspace $Y$ is called \emph{locally 1-complemented} in $X$ if
for every finite dimensional subspace $E$ of $X$
and every $\varepsilon > 0$ there exists a linear
operator $u: E \to Y$ such that
$u(e) = e$ for all $e \in E \cap Y$ and
$\|u\| \le 1 + \varepsilon$. Fakhoury
\cite[Th\'eor\`eme~2.14]{MR0348457} proved that $Y$ is an
ideal in $X$ precisely when it is locally 1-complemented in $X$.

Following \cite{ALN02} we say that $Y$ is an
\emph{almost isometric ideal (ai-ideal)}  in $X$ if $Y$ is locally
1-complemented in $X$ in such a way that the operator $u: E \to Y$
is an almost isometry, i.e. in addition to the above we have
$(1+\varepsilon)^{-1}\|e\| \le \|u(e)\| \le (1+\varepsilon)\|e\|$
for all $e \in E$.
The fact that every Banach space is an ai-ideal in its bidual is commonly
referred to as \emph{the Principle of Local Reflexivity (PLR)}.

\begin{lem}\label{lem:*asq-ai-ideal}
  If $X$ is (L)ASQ and $Y$ is an ai-ideal in $X$
  then $Y$ is (L)ASQ.
\end{lem}

\begin{proof}
  We only show the ASQ case.
  Let $y_1,y_2,\ldots,y_N \in S_Y$ and $0 < \varepsilon < 1$.
  Find $x \in S_X$ such that $\|y_i - x\| \le 1 + \varepsilon/4$.
  Now, choose an $\varepsilon/4$-isometry $u: E \to Y$ such that
  $u$ is the identity on
  $E \cap Y$ where $E = \linspan\{(y_j)_{j=1}^N, x\}$.
  Define $z = u(x)/\|u(x)\|$. Then $z \in S_Y$
  and $\|u(x) - z\| = |\|u(x)\|-1| \le \varepsilon/4$ and
  \begin{align*}
    \|y_i - z\| &\le \|u(y_i - x)\| + \|u(x)-z\| \le
    (1+\frac{\varepsilon}{4})(1+\frac{\varepsilon}{4}) +
    \frac{\varepsilon}{4}
    \le 1+\varepsilon.
  \end{align*}
  Thus $Y$ is ASQ by Proposition~\ref{prop:lasNas-1dim}.
\end{proof}

If $Y$ is an ideal in $X$ with an
ideal projection $P$ on $X^*$ which for every $x^* \in X^*$ satisfies
$\|x^*\| = \|Px^*\| + \|x^* - Px^*\|$, then $Y$ is said to be an
\emph{M-ideal} in $X$ ($P$ is called the M-ideal projection on
$X^*$). If $X$ is an M-ideal in $X^{**}$, then $X$ is said to
be \emph{M-embedded}.
For M-ideals we often get ASQ for free.

\begin{thm}\label{thm:mid-sc-sq}
  Let $Y$ be a proper subspace of a non-reflexive Banach space $X$.
  If $Y$ is both an M-ideal and an ai-ideal in $X$, then $Y$ is
  ASQ.
\end{thm}

\begin{proof}
  Let $\varepsilon > 0$ and choose $0 < \delta < 1$ with
  $(1+\delta)^2(1 + 3\delta(1+\delta)^2) < 1 + \varepsilon$.
  Write $X^{**} = (PX^*)^{\perp} \oplus_\infty Y^{\perp\perp}$.
  This is possible as $Y$ is an M-ideal in $X$ and
  thus $X^* = P(X^*) \oplus_1 Y^{\perp}$
  ($P$ denotes here the M-ideal projection on $X^*$).
  Let $y_1,y_2,\ldots,y_N \in S_Y$ and $z \in S_{(PX^*)^{\perp}}$,
  and put $E = \linspan\{(y_i)_{i=1}^N, z\} \subset X^{**}$.
  Use the PLR to find a $\delta$-isometry $v: E \to X$
  which is the identity on $E \cap X$.
  Further, put $F = v(E) \subset X$ and use that $Y$ is an
  ai-ideal in $X$ to find a $\delta$-isometry $u:F \to Y$ which
  is the identity on $F\cap Y$.
  Now with $y = uv(z)/\|uv(z)\| \in S_Y$ we use $uv(y_i) = y_i$ to get
  \begin{align*}
    \|y_i - y\| &= \|y_i - \frac{uv(z)}{\|uv(z)\|}\|  \le (1 +
    \delta)^2\|y_i - \frac{z}{\|uv(z)\|}\| \\
                & \le (1 + \delta)^2(\|y_i - z\| +
                \|z - \frac{z}{\|uv(z)\|}\|) < 1 + \varepsilon
  \end{align*}
  since
  \begin{align*}
    \| z - \frac{z}{\|uv(z)\|} \|
    &= \frac{1}{\|uv(z)\|}|1 - \|uv(z)\|| \\
    &\le (1+\delta)^{2}
    ( | 1 - \|v(z)\| | + | \|v(z)\| - \|uv(z)\| | ) \\
    &\le (1+\delta)^2 (\delta + \delta(1+\delta)) \le 3\delta(1+\delta)^2.
  \end{align*}
  Using Proposition \ref{prop:lasNas-1dim} we are done.
\end{proof}

Since every Banach space is an ai-ideal in its bidual
by the PLR we immediately
have the following corollary.

\begin{cor}\label{cor:mem-asq}
  Non-reflexive M-embedded spaces are ASQ.
\end{cor}

The following spaces are examples of M-embedded spaces: $c_0(\Gamma)$
(for any set $\Gamma$), $\mathcal K(H)$
of compact operators on a Hilbert space $H$, and $C(\mathbb T)/A$
where $\mathbb T$ denotes the unit circle and $A$ the disk algebra.
(For more examples see Chapter~III in \cite{HWW}.)
From Example \ref{ex1:c_0} the space $c_0(\ell_1)$ is ASQ. However,
this space contains a copy of $\ell_1$ and therefore can not be M-embedded
(\cite[Theorems~3.4.a and 3.5]{MR735420}). Thus the class of ASQ
spaces properly contains the class of M-embedded spaces.

\section{Stability}
\label{sec:stability}
We start this section by recalling the notion of an absolute
sum of a family of Banach spaces. Our goal is to show that  LASQ and
WASQ spaces are stable under absolute sums (see Propositions
\ref{prop:stab-was} and \ref{prop: stab loct and lasq}).
It turns out that locally and
weakly octahedral Banach spaces are stable by forming absolute sums too
(see Propositions \ref{prop: stab loct and lasq} and \ref{prop:stab-oct}).

\begin{defn}
  Let $I$ be a non-empty set and let
  $E$ be a $\mathbb{R}$-linear subspace of $\mathbb{R}^I$.
  An \emph{absolute norm} on $E$ is a complete norm
  $\|\cdot\|_E$ satisfying
  \begin{enumerate}
  \item\label{item:1}
    Given $(a_i)_{i \in I},(b_i)_{i \in I} \in \mathbb{R}^I$ with
    $|b_i| \le |a_i|$ for every $i \in I$,
    if $(a_i)_{i \in I} \in E$, then $(b_i)_{i \in I} \in E$ with
    $\|(b_i)_{i \in I}\|_E \le \|(a_i)_{i \in I}\|_E$.
  \item\label{item:2}
    For every $i \in I$, the function $e_i : I \to \mathbb{R}$
    given by $e_i(j) = \delta_{ij}$ for $j \in I$, belongs
    to $E$ and $\|e_i\|_E = 1$.
  \end{enumerate}
\end{defn}

Let $E\subset \mathbb{R}^I$ with an absolute norm.
Then $\ell_1(I) \subseteq E \subseteq \ell_\infty(I)$, and
$E$ can be viewed as a K{\"o}the function
space (and hence a Banach lattice) on the space
$(I,\mathcal{P}(I),\mu)$, where $\mathcal{P}(I)$ is the power
set of $I$ and $\mu$ is the counting measure on $I$.
It is known that $E$ is order continuous if and only if $E$ does not
contain an isomorphic copy of $\ell_\infty$ if and only if
$\linspan\{e_i:i \in I\}$ is dense in $E$.

The K{\"o}the dual $E'$ of a Banach space $E \subset \mathbb{R}^I$
with absolute norm is the linear subspace
of $\mathbb{R}^I$ defined by
\begin{equation*}
  E':= \left\{(a_i)_{i \in I} \in \mathbb R^I: \sup \sum_{i \in I} |a_ib_i| < \infty,
  {(b_i)_{i \in I} \in B_E} \right\}.
\end{equation*}

It is not hard to see that
\begin{equation*}
 \|(a_i)_{i \in I}\|_{E'}:= \sup \left\{ \sum_{i \in I} |a_ib_i|
 : {(b_i)_{i \in I} \in B_E} \right\}
\end{equation*}
defines an absolute norm on $E'$.
Every $(b_i)_{i \in I} \in E'$ defines a functional on $E$ by
\begin{equation*}
  (a_i)_{i \in I} \to
  \sum_{i \in I} b_i a_i.
\end{equation*}
This induces an embedding $E' \to E^*$ which is easily seen to be  linear and isometric. If $\linspan\{e_i:i \in I\}$ is dense in $E$ then the embedding $E' \to E^*$ is surjective, and so $E'$ and $E^{\ast}$ can be identified.

Now, if $(X_i)_{i \in I}$ is a family of Banach spaces
we put
\begin{equation*}
  [\oplus_{i \in I} X_i]_E := \{ (x_i)_{i \in I} \in \Pi_{i \in
    I} X_i: (\|x_i\|)_{i \in I} \in E \}.
\end{equation*}
It is clear that this defines a subspace of the product space
$[\oplus_{i \in I} X_i]_E$ which becomes a Banach space when given the norm
\begin{equation*}
  \|(x_i)_{i \in I}\|:= \|(\|x_i\|)_{i \in I}\|_E, \qquad
  (x_i)_{i \in I} \in [\oplus_{i \in I} X_i]_E.
\end{equation*}
This Banach space is said to be the \emph{absolute sum of the
family $(X_i)_{i \in I}$ with respect to $E$}. Every $(x^*_i)_{i \in I} \in [\oplus_{i \in I} X^*_i]_{E'}$ defines a functional on $[\oplus_{i \in I} X_i]_E$ by
\begin{equation*}
  (x_i)_{i \in I} \to
  \sum_{i \in I} x^*_i(x_i).
\end{equation*}
This embedding is isometric and is onto if
$\linspan\{e_i:i \in I\}$ is dense in $E$.

Putting $I = \mathbb N$ and $E = \ell_p(I)$ it is clear that for $1
\le p \le \infty$ the $\ell_p$
sum ($c_0$ sum if $p=\infty$) of a family of Banach spaces $(X_i)_{i
  \in I}$ is an absolute sum with
respect to $E$ (for which ${[\oplus_{i \in I}
X^*_i]_{E'}} = {[\oplus_{i \in I}
X_i]_E}^{\ast}$ as $\linspan\{e_i:i \in I\}$ is dense in $\ell_p(I)$ in this
case). In \cite[Propositions~3.4 and 3.7]{HLP} it was proved that
locally and weakly octahedral spaces are stable by taking $\ell_p$
sums of two Banach spaces. This can also be obtained from
Propositions \ref{prop: stab loct and lasq} and \ref{prop:stab-oct} below.
First we show that WASQ is stable by taking absolute sums.

\begin{prop} \label{prop:stab-was}
  Let $E$ be a subspace of $\mathbb R^{I}$ with an absolute norm such that
  $\linspan\{e_i:i \in I\}$ is dense in $E$.
  If $(X_i)_{i \in I}$ is a family of Banach spaces which
  are WASQ, then $X=(\oplus_{i \in I} X_i)_E$ is WASQ.
\end{prop}

\begin{proof}
  Let $x=(x_i)_{i \in I}\in S_X$.
  Our task is to find a weakly-null sequence $(y_n)
  \subset S_X$ such that
  \begin{equation*}
    \|x\pm y_n\|_E \to 1.
  \end{equation*}
  Since $\linspan\{e_i: i \in I\}$ is dense
  in $E$ we may assume that $J = \{i \in I : x_i \neq 0\}$
  is finite.
  By assumption, for every $i \in J$, there exist weakly-null sequences
  $(y^n_i) \subset S_{X_i}$
  such that
  \begin{equation*}
    \big\|\frac{x_i}{\|x_i\|}\pm y^n_i\big\| \le 1 + n^{-1}.
  \end{equation*}
  If we let $y_n=(\|x_i\|y^n_i)_{i\in I}$, with $\|x_i\|y_i^n = 0$
  for $i \not\in J$, then $\|y_n\|_E=1$ and
  \begin{equation*}
    \|x \pm y_n\|_E
    = \big\|\big(\|x_i \pm \|x_i\|y^n_i \|\big)_{i\in I}\big\|_E
    \le \big\|\big((1+n^{-1})\|x_i\|\big)_{i \in I}\big\|_E
    \le 1 + n^{-1}.
  \end{equation*}
  Let $x^* = (x_i^*)_{i\in I} \in X^*$.
  Since $x^*(y_n) = \sum_{i \in J} x_i^* (y_i^n)$
  and $J$ is finite we get $x^*(y_n) \to 0$ and
  thus $(y_n)$ is weakly-null.
\end{proof}

In the proof above we assumed that our unit vector had finite
support in order to get a weakly-null sequence.
For absolute sums of LASQ and locally octahedral spaces
this is not necessary and the first part of the proof above
gives the following proposition.

\begin{prop}\label{prop: stab loct and lasq}
  Let $I$ be a set, $E$ a subspace of $\mathbb R^I$ with an absolute
  norm, and $(X_i)_{i \in I}$ a family of Banach spaces which are
  locally octahedral (resp. LASQ).
  Then their absolute sum $X=(\oplus_{i \in I} X_i)_E$
  is locally octahedral (resp. LASQ).
\end{prop}

For absolute sums of weakly octahedral spaces we have
to work a bit harder.

\begin{prop} \label{prop:stab-oct}
  Let $I$ be a set, $E$ a subspace of $\mathbb R^I$ with an absolute
  norm such that $\linspan\{e_i:i \in I\}$ is dense in $E$, and
  $(X_i)_{i \in I}$ a family of Banach spaces which are weakly
  octahedral. Then their absolute sum $X=(\oplus_{i \in I} X_i)_E$ is
  weakly octahedral.
\end{prop}

\begin{proof}
  Let $\varepsilon>0$, let $x_1=(x^1_i)_{i\in I}, x_2=(x^2_i)_{i\in I}, \ldots,
  x_N=(x^N_i)_{i\in I}\in S_X$, and $x^*=(x^*_i)_{i\in I}\in
  B_{X^*}$. Our task here is to find $y\in S_X$ such that
  for all $t>0$ and $k=1,2,\ldots,N$
  \begin{equation*}
    \|x_k+ty\|_E\geq (1-\varepsilon)(|\sum_{i\in I}x^\ast_i(x^k_i)|+t).
  \end{equation*}
  Let $z^\ast_i=\frac{x^\ast_i}{\|x^\ast_i\|}$ if $x^\ast_i\neq 0$ and
  $z^\ast_i=0$ otherwise. By the weak octahedrality of $X_i$, for every
  $i\in I$, there exists a $y_i \in S_{X_i}$ such that for all $t>0$ and
  $k=1,2,\ldots,N$
  \begin{equation}\label{WOH}
    \big\|\frac{x^k_i}{\|x^k_i\|}+ty_i\big\| \geq
    (1-\varepsilon/2)\big(\frac{|z^\ast_i(x^k_i)|}{\|x^k_i\|}+t\big).
  \end{equation}
  If $x^k_i=0$ for some $i\in I$, then take $y_i$ to be any element from
  $S_{X_i}$. Now (\ref{WOH}) implies that
  for all $t>0$ and $k=1,2,\ldots,N$
  \begin{equation*}
    \|x^k_i+ty_i\| \geq
    (1-\varepsilon/2)(\frac{|x^\ast_i(x^k_i)|}{\|x^\ast_i\|}+t).
  \end{equation*}
  Since $\|x^*\| = \|(\|x^*_i\|)_{i\in I}\|_{E^*}\leq 1$, there is a list of
  reals $(\alpha_i)_{i\in I}\subset \mathbb{R}$ such that
  $\|(\alpha_i)_{i\in I}\|_E=\|(|\alpha_i|)_{i\in I}\|_E=1$ and
  \begin{equation*}
    \sum_{i \in I} \|x^\ast_i\| \cdot |\alpha_i|
    >
    \|x^*\| ( 1-\frac{\varepsilon/2}{1-\varepsilon/2} ).
  \end{equation*}
  We take $y=(|\alpha_i| y_i)_{i\in I}\in S_X$ to get
  \begin{align*}
    \|x^*\| \|x_k+ty\|_E
    &\geq \sum_{i\in I}\|x^\ast_i\|\cdot\|x^k_i+|\alpha_i|ty_i \| \\
    &\geq (1-\varepsilon/2)
    \sum_{i\in I}\|x^\ast_i\|
    \left(\frac{|x^\ast_i(x^k_i)|}{\|x^\ast_i\|}+|\alpha_i|t\right) \\
    &\geq (1-\varepsilon/2)
    \left( |\sum_{i\in I}x^\ast_i(x^k_i)|
      + \|x^*\| (1-\frac{\varepsilon/2}{1-\varepsilon/2}) t\right) \\
    &\geq
    (1-\frac{\varepsilon/2}{1-\varepsilon/2})(1-\varepsilon/2) \|x^*\|
    \left(
      |\sum_{i \in I} x^\ast_i(x^k_i) | + t
    \right) \\
    &= \|x^*\|(1-\varepsilon)\left(|\sum_{i\in I}x^\ast_i(x^k_i)|+t \right).
  \end{align*}
  Dividing both sides by $\|x^*\|$ we get the desired inequality.
\end{proof}

We have seen that for a sequence of non-trivial Banach spaces $(X_i)$ the space $c_0(X_i)$ is always ASQ. Similarly $\ell_1(X_i)$ is always octahedral.

Note that $X\oplus_p Y$, $1<p<\infty$, can never be ASQ,
because it fails the SD2P
(see \cite[Theorem~3.2]{ABGLP13} or \cite[Theorem~1]{HL}).
But even though the SD2P property is stable by forming $\ell_1$ sums (see
\cite[Theorem~2.7 (iii)]{ALN01}), it turns out that the $\ell_1$ sum of
Banach spaces can never be ASQ.

\begin{lem}
Let $X$ and $Y$ be nontrivial Banach spaces. Then $X\oplus_1 Y$ is never ASQ.
\end{lem}
\begin{proof}
Let $Z=X\oplus_1 Y$, $x\in S_X$, and $y\in S_Y$. Consider norm $1$
elements $z_1=(-\frac{1}{3}x, \frac{2}{3}y)$ and $z_2=(\frac{2}{3}x,
-\frac{1}{3}y)$. Assume for contradiction that there is a
$w=(w_x, w_y)\in S_Z$ with $\|z_i\pm w\|\leq 1+\frac{1}{9}$. Then
\begin{align*}
\|w_x\|+\|\frac{2}{3}y\| &\leq \frac{1}{2}\Big(\|-\frac{1}{3}x+
w_x\|+\|\frac{2}{3}y+ w_y\|+\|\frac{1}{3}x+ w_x\|+\|\frac{2}{3}y-
w_y\|\Big)\\
& \leq \max\{\|z_1+ w\|,\|z_1- w\| \}\leq 1+\frac{1}{9}
\end{align*}
so that $\|w_x\|\leq \frac{1}{3}+\frac{1}{9}$.
Similarly $\|w_y\|\leq \frac{1}{3}+\frac{1}{9}$.
We get $\|w\| < 1$ which is a contradiction.
\end{proof}

\begin{cor}\label{p-sum of LASQ spaces implies X is LASQ}
  Let $X$ and $Y$ be nontrivial Banach spaces and $1\leq p<\infty$.
  \begin{enumerate}
  \item $X\oplus_p Y$ is LASQ if and only if $X$ and $Y$ are LASQ.
  \item $X\oplus_p Y$ is WASQ if and only if $X$ and $Y$ are WASQ.
  \end{enumerate}
\end{cor}

\begin{proof}
(i).
One direction is Proposition~\ref{prop: stab loct and lasq}.
Let us show that $X$ is LASQ whenever $X\oplus_p Y$ is.

The function $f(x) = x^{1/p}$ is uniformly continuous on $[0,2]$
so given $\varepsilon > 0$ there exists $\delta > 0$
such that $|f(x)-f(y)|\le\varepsilon$ whenever $|x-y| \le \delta$.
Also the function $g(x) = x^p$ is continuous at $x=1$
so there exists $\eta > 0$
such that $|g(1)-g(y)|\le \delta$ whenever $|1-y| \le \eta$.

Assume $x \in S_X$.
Since $X \oplus_p Y$ is LASQ
there exists $(u,v) \in S_{X \oplus_p Y}$ such that
\begin{equation*}
  \|(x,0) \pm (u,v)\|_p
  = \left(\|x \pm u\|^p + \|v\|^p\right)^{1/p} \le 1+ \eta.
\end{equation*}
(Note that $u \neq 0$, else $\|(x,v)\| = 2^{1/p} > 1 +\varepsilon$.)
We have (since $t \mapsto t^p$ is increasing)
\begin{equation*}
  \|x\pm u\|^p + \|v\|^p \le (1+\eta)^p \le 1^p + \delta
  = 1 + \delta
\end{equation*}
hence
\begin{equation*}
  \|x \pm u\|^p \le
  1 + \delta - \|v\|^p
  = \|u\|^p + \|v\|^p - \|v\|^p + \delta
  = \|u\|^p + \delta.
\end{equation*}
Taking $p$-th roots we get
\begin{equation*}
  \|x \pm u\| \le \|u\| + \varepsilon
\end{equation*}
since $|\|u\|^p+\delta-\|u\|^p| = \delta$.
Let $z = u/\|u\|$. Then
\begin{equation*}
  \|x \pm z\| \le \|x \pm u\|
  + \|z - u\| \le \|u\| + \varepsilon + 1-\|u\|
  = 1+\varepsilon.
\end{equation*}

(ii). The proof is similar to (i). Indeed, for $\varepsilon_n=\frac{1}{n}$ find the sequence $\eta_n$ and observe that if a sequence $(u_n,v_n)$ converges weakly to $(0,0)$ in $X\oplus_p Y$, then $u_n$ converges weakly to $0$ in $X$.
\end{proof}

We end this section by showing that for finite $\ell_\infty$ sums we only need to assume that only one of the spaces is LASQ, WASQ or ASQ.

\begin{prop}\label{prop:*asq-inftysum}
Let $X$ and $Y$ be nontrivial Banach spaces.
\begin{enumerate}
\item $X\oplus_\infty Y$ is LASQ if and only if either $X$ or $Y$ is LASQ.

\item $X\oplus_\infty Y$ is WASQ if and only if either $X$ or $Y$ is WASQ.

\item $X\oplus_\infty Y$ is ASQ if and only if either $X$ or $Y$ is ASQ.

\end{enumerate}
\end{prop}

\begin{proof}
Let $Z = X \oplus_\infty Y$.
We will only prove the ASQ case -- the others will follow similarly.

Suppose that $Z$ is ASQ. Let $x_1,x_2,\ldots, x_N \in
S_X$ and $y_1,y_2,\ldots, y_N \in S_Y$. Then $(x_i,y_i)$ is in $S_Z$ for
every $i=1,2,\ldots,N$ and by our assumption there is a sequence $z_n =
(u_n,v_n)$ in $B_Z$ such that $||(x_i,y_i) \pm
(u_n,v_n)||\to 1$ for every $i=1,2,\ldots,N$ and $||z_n||
\to 1$. Since $||z_n|| \to 1$ there is a
subsequence such that either $||u_n|| \to 1$ or
$||v_n||\to 1$. Thus one of the spaces $X$ or $Y$ must be
ASQ.

Suppose now that $X$ is ASQ.
Let $z_i=(x_i, y_i)\in S_Z$ for $i=1,2,\ldots,N$ and
let $\varepsilon > 0$.
We may assume that $x_i \neq 0$ for $i=1,2,\ldots,N$.
Using Proposition~\ref{prop:lasNas-1dim}
we can find a $u \in S_X$
$\|\frac{x_i}{\|x_i\|} - u\| \le 1+ \varepsilon$
for every $i=1,2,\ldots,N$. Put $z=(u,0)$. Then
\begin{equation*}
  \|z_i-z\| \le \|x_i - u\|
  =
  \left\| \; \|x_i\| (\frac{x_i}{\|x_i\|} - u)
    + (1-\|x_i\|)u\right \|
  \le 1+\varepsilon
\end{equation*}
for every $i=1,2,\ldots,N$ and $Z$ is ASQ.
\end{proof}

\section{Connection with the IP}
\label{sec:connection-with-ip}
In this section we explore the connection between
ASQ spaces and the intersection property introduced
in \cite{BeHa} (see also \cite[Chapter~II.4]{HWW}).
For a set $I \subset \mathbb{R}^+$ we will
use the notation $B_I = \{x \in X : \|x\| \in I \}$.
So for example $B_X = B_{[0,1]}$, $S_X = B_{\{1\}}$ and
$B_X \setminus S_X = B_{[0,1)}$.

A Banach space $X$ has the \emph{intersection property} (IP)
if for every $\varepsilon > 0$ there exist $x_1,x_2,\ldots,x_N$
in $X$ with $\|x_i\|<1$, $i=1,2,\ldots,N$, such that
if $y \in X$ with
$\|x_i - y\| \le 1$, for every $i=1,2,\ldots,N$, then $\|y\| \le \varepsilon$.
If $X$ fails the IP, then for some $0 < \varepsilon < 1$
we have $\gamma(\varepsilon) = 1$ where
\begin{equation*}
  \gamma(\varepsilon) =
  \sup_{x_1,x_2,\ldots,x_n \in B_{[0,1)}} \inf_{y \in B_{(\varepsilon,1]}}
  \max_i \|x_i -y\|.
\end{equation*}
We will say that $X$ \emph{$\varepsilon$-fails the IP}
if $\gamma(\varepsilon) = 1$.
This is very similar to the index $\alpha(\varepsilon)$,
$\varepsilon \in [0,1]$, defined by Maluta
and Papini in \cite{MR1223652}. Here are two
equivalent definitions of $\alpha(\varepsilon)$
(see Proposition~3.3 in \cite{MR1223652})
\begin{align*}
  \alpha(\varepsilon)
  &=
  \sup_{x_1,x_2,\ldots,x_n \in S_X} \inf_{y \in B_{[\varepsilon,1]}}
  \max_i \|x_i -y\| \\
  &=
  \sup_{x_1,x_2,\ldots,x_n \in B_{[0,1)}} \inf_{y \in B_{[\varepsilon,1]}}
  \max_i \|x_i -y\|.
\end{align*}
It is clear that $\alpha(\varepsilon)$ is monotone,
$\alpha(0) = 1$, and that for $\varepsilon = 1$
we get the thinness index so that $\alpha(1) = t(X) = 1$
if and only if $X$ is ASQ.

\begin{prop}\label{thm:epsfailsIPandas}
  A Banach space $X$ is ASQ if and only if $X$ $\varepsilon$-fails
  the IP for all $0 < \varepsilon < 1$.
\end{prop}

\begin{proof}
  Assume $X$ is ASQ and $0 < \varepsilon < 1$.
  Since
  \begin{equation*}
    t(X) = \alpha(1) \ge
    \gamma(\varepsilon)
    \ge \alpha(\varepsilon) \ge \alpha(0) = 1
  \end{equation*}
  we get $\gamma(\varepsilon) = 1$.

  Conversely. Let $x_1,x_2,\ldots,x_N \in S_X$ and $\varepsilon > 0$.
  Let $z_i = (1+\varepsilon)^{-1}x_i$.
  Since $X$ $r$-fails the IP for $r = 1-\varepsilon$,
  there is
  $y \in B_{(r,1]}$ with
  $\max_{i}\|z_i - y\| \le 1 + \varepsilon$.
  Then
  \begin{equation*}
    \|x_i - \frac{y}{\|y\|} \| \le \|x_i - z_i\| + \|z_i - y\|
    + (1-\|y\|)
    \le 1 + 3 \varepsilon.
  \end{equation*}
  From Proposition~\ref{prop:lasNas-1dim}
  we conclude that $X$ is ASQ.
\end{proof}

\begin{rem}
  Harmand and Rao, Theorem~1.7 in \cite{HaRao}, showed that
  every Banach space $X$ containing $c_0$ can be renormed
  to fail the IP.
  In Theorem~\ref{thm:c0-asq} we saw that if $X$ is separable it
  can even be renormed to be ASQ.
  From Theorem~\ref{thm:epsfailsIPandas} we
  see that this is a strengthening of Harmand and Rao's result
  for separable spaces.
\end{rem}

\begin{exmp}
  For $r > 1$ define
  \begin{equation*}
    X_r = \{ f \in C[0,1] : f(0) = r f(1) \}.
  \end{equation*}
  $X_r$ is a non-LASQ space which
  $\varepsilon$-fails the IP for all
  $\varepsilon \le 1 - \frac{1}{r}$, but
  does not $\varepsilon$-fail the IP
  for any $\varepsilon > 1 - \frac{1}{r}$.

  Let $f(x) = (1-x) + \frac{1}{r}x$.
  Let $g \in X_r$ with $\|g\|=\varepsilon$.
  Find $x_0 \in [0,1]$ such that $|g(x_0)|=\varepsilon$ then
  \begin{equation*}
    \frac{1}{r} + \varepsilon
    \le \max_{\pm} |f(x_0) \pm g(x_0)|
    \le \|f \pm g\|.
  \end{equation*}
  With $\varepsilon=1$ this shows that $X_r$ is not LASQ
  since $\|f \pm g\|$ is bounded away from $1$.
  It also shows that $\alpha(\varepsilon) \ge \frac{1}{r} +
  \varepsilon$.

  For $f \in X_r$ with $\|f\| < 1$ we have
  $|f(1)| < \frac{1}{r}$ (if not $|f(0)| \ge 1$).

  Let $f_1,f_2,\ldots,f_n \in X_r$ with $\|f_i\| < 1$.
  Find an interval $(a,1)$ such that $|f_i(x)| < \frac{1}{r}$
  for $x \in (a,1)$.
  Let $g \in X_r$ with $\supp g \subset (a,1)$
  then $\|f_i + g\| < \frac{1}{r} + \|g\|$ and there
  exists $\delta > 0$ such that $\|f_i + g\| + \delta
  \le \frac{1}{r} + \|g\|$.
  If we choose $g$ as above with $\|g\| = 1 - \frac{1}{r} + \delta$
  then $\max_i \|f_i + g\| \le 1$ and $\|g\| > 1 - \frac{1}{r}$.
  Hence $X_r$ $(1-\frac{1}{r})$-fails IP.

  By the above we also have for $\eta > 0$
  \begin{equation*}
    \gamma((1-\frac{1}{r})+\eta)
    \ge
    \alpha((1-\frac{1}{r})+\eta)
    \ge \frac{1}{r} + (1-\frac{1}{r})+\eta
    = 1 + \eta > 1.
  \end{equation*}
\end{exmp}

\begin{exmp}\label{exmp:asqNotMid}
  The space $X = \ell_\infty(C_{\Sigma}(S^m))$ is ASQ.
  Here $S^m$ is the Euclidean sphere in $\mathbb R^{m+1}$ and
  \begin{equation*}
    C_\Sigma(S^m) = \{f \in C(S^m): f(s) = -f(-s) \;
    \mbox{for all}\; s \in S^m\}.
  \end{equation*}
  $X$ is not a $c_0$-sum of ASQ-spaces nor M-embedded
  (see \cite[Example~II.4.6, p.~78]{HWW}), but
  a small adjustment to the proof of in
  \cite[Proposition~II.4.2~(h), p.~76]{HWW} shows that
  $X$ $\varepsilon$-fails the IP for every $0 < \varepsilon < 1$.
\end{exmp}

Next we will show that every ASQ space
contains a separable subspace which is ASQ.
The basic idea for the next proof goes back
to Theorem~4.4 in Lindenstrauss' memoir \cite{MR0179580}.

\begin{prop}\label{prop:as-sep}
  If $X$ is ASQ, then for every separable subspace
  $Y$ of $X$ there exists a separable subspace
  $Z$ with $Y \subset Z \subset X$ and
  $Z$ is ASQ.
\end{prop}

\begin{proof}
  Let $Y \subset X$ be separable.
  With $\varepsilon_n = 2^{-n}$ and
  $Y_0 = Y$ we construct a sequence of separable
  subspaces $(Y_n)$ inductively.

  Let $A_{n}$ be a countable dense set in $S_{Y_n}$.
  For each finite family $G$
  in $A_{n}$ find $y_G$ in $S_X$ such that
  $\|x \pm y_G\| < 1 + \varepsilon_{n}$ for all $x \in G$.
  Let $Y_{n+1}$ be the closure of $\linspan\{Y_n,(y_G)\}$.
  $Y_{n+1}$ is separable since $Y_n$ is.

  Define $Z = \overline{\cup Y_n }$. $Z$ is separable and ASQ.
  Let $z_1,z_2,\ldots,z_N \in S_Z$ and $\varepsilon > 0$.
  Choose $k$ such that $\varepsilon_k < \varepsilon/2$
  and find $x_1,x_2,\ldots,x_N$ in $A_k$ with
  $\|x_i - z_i\| < \varepsilon/2$.
  Then there exists a $y$ in $S_{Y_{k+1}} \subset S_Z$
  with $\|z_i \pm y\| < 1+\varepsilon$
  for $i=1,2,\ldots,N$.
\end{proof}

In \cite{BeHa} Behrends and Harmand asked if dual spaces always have the IP.
None of the examples of ASQ spaces we have seen are dual spaces.
We ask:
\begin{quest}
  Can the dual $X^*$ of a Banach space $X$ be ASQ?
\end{quest}

\begin{rem}
  In Remark~2a page~289 in \cite{HaRao} Harmand and Rao
  noted the following partial answer to
  the question about the IP:
  If $X^*$ is such that for any separable subspace
  $Y$ of $X^*$ there is separable subspace $Z$
  with $Y \subset Z \subset X^*$ and $Z$ complemented
  in $X^*$, then $X^*$ has the IP.
  (The assumption is satisfied if
  e.g. $X^*$ is weakly compactly generated.)
  Their arguments works also for ASQ spaces
  and show that an ASQ space can never be
  a subspace of a weakly compactly generated
  dual space.
\end{rem}

\end{document}